\documentclass[twocolumn,conference,10pt]{IEEEtran}
\usepackage{xcolor,soul,amssymb,amsthm,amsmath,dsfont} 
\usepackage{array}
\usepackage{url}
\usepackage{amsmath}
\usepackage{amssymb}
\usepackage[mathscr]{euscript}
\usepackage{mathtools}
\usepackage[utf8]{inputenc}
\usepackage[english]{babel}
\usepackage{amsthm}
\usepackage{bbm}
\usepackage{subcaption}
\usepackage{algorithm,algorithmic}
\usepackage{array}
\usepackage[pdftex]{graphicx}
\graphicspath{{figures/}{./}}
\usepackage{tikz,pgfplots}

\pgfplotsset{compat=1.16} 
\tikzstyle{vertex} = [fill,shape=circle,node distance=30pt]
\tikzstyle{edge} = [fill,opacity=.6,fill opacity=.5,line cap=round, line join=round, line width=10pt]
\tikzstyle{elabel} =  [fill,shape=circle,node distance=30pt,fill opacity=.9]

\pgfdeclarelayer{background}
\pgfsetlayers{background,main}

\ifCLASSINFOpdf
\else
\fi

\newcommand{\black}{\color{black}}

\newtheorem{theorem}{Theorem}

\newtheorem{proposition}[theorem]{Proposition}
\newtheorem{problem}{Problem}

\newtheorem{assumption}{Assumption}

\newcommand{\calC}{\mathcal{C}}

\newcommand{\calL}{\mathcal{L}}
\newcommand{\bfP}{\mathbf{P}}
\newcommand{\bfQ}{\mathbf{Q}}

\DeclareMathOperator*{\argmax}{arg\,max}
\DeclareMathOperator*{\argmin}{arg\,min}

\begin{document}
\bstctlcite{}
\title{Network Learning with Directional Sign Patterns}
\author{Anqi Dong, Can Chen, Tryphon T.~Georgiou\\
\thanks{Anqi Dong and Tryphon T.~Georgiou are with the Department of Mechanical and Aerospace Engineering, University of California, Irvine, CA 92697, USA. \texttt{\{anqid2,tryphon\}@uci.edu}}
\thanks{Can Chen is with the School of Data Science and Society and the Department of Mathematics, University of North Carolina at Chapel Hill, Chapel Hill, NC 27599, USA. \texttt{canc@unc.edu}}}
\markboth{Dong \MakeLowercase{\textit{et al.}}: Network Learning with Directional Sign Patterns
}{Dong \MakeLowercase{\textit{et al.}}: Network Learning with Directional Sign Patterns}
\maketitle

\begin{abstract}
Complex systems can be effectively modeled via graphs that encode networked interactions, where relations between entities or nodes are often quantified by signed edge weights, e.g., promotion/inhibition in gene regulatory networks, or encoding political of friendship differences in social networks. However, it is often the case that only an aggregate consequence of such edge weights that characterize relations may be directly observable, as in protein expression of in gene regulatory networks. Thus, learning edge weights poses a significant challenge that is further exacerbated for intricate and large-scale networks. 
In this article, we address a model problem to determine the strength of sign-indefinite relations that explain marginal distributions that constitute our data. To this end, we develop a paradigm akin to that of the Schr\"odinger bridge problem and an efficient Sinkhorn type algorithm (more properly, Schr\"odinger-Fortet-Sinkhorn algorithm) that allows fast convergence to parameters that minimize a relative entropy/likelihood criterion between the sought signed adjacency matrix and a prior. The formalism that we present represents a novel generalization of the earlier Schr\"odinger formalism in that marginal computations may incorporate weights that model directionality in underlying relations, and further, that it can be extended to high-order networks -- the Schr\"odinger-Fortet-Sinkhorn algorithm that we derive is applicable all the same and allows geometric convergence to a sought sign-indefinite adjacency matrix or tensor, for high-order networks.
We demonstrate our framework with synthetic and real-world examples.
\end{abstract}


\IEEEpeerreviewmaketitle

\section{Introduction}
Complex interactions between multiple subsystems can often be encoded in a networked structure that models respective relations \cite{mesbahi2010graph,newman2018networks}. These relations may be binary as in ordinary graphs with vertices and edges, trinary and so on, as in higher-order relations. The strength of respective interactions may be assigned as a weight in corresponding edges, or their higher-order analogues. Yet, contrary to typical accounts in the vast literature on the subject, interactions may be sign-indefinite, signifying e.g., promotion/inhibition or friend/unfriend in biological and social networks, respectively. Moreover, perceived aggregate effects of such relations may have an added directional bias. In many such examples, acquiring nodal information is relatively straightforward, while directly and precisely learning edge weights presents a significant challenge. Numerous computational methods have been proposed to learn and quantify edge weights in networks, including cross-correlation \cite{garofalo2009evaluation}, Granger causality \cite{basu2015network}, mutual information \cite{villaverde2014mider}, system identification \cite{chiuso2019system}, and graph transformer neural networks \cite{yun2019graph}. However, most of these methods require adequate time-series data on node states, which might not be easily accessible in general.

In our recent work \cite{dong2023negative}, we presented a framework for identifying sign-distinguishable edge weights by utilizing prior knowledge of the sign-indefinite structure, encoded in a sign-indefinite ``prior'' adjacency matrix, and nodal statistics that constitute data. The formalism involves minimizing a suitable relative entropy functional between sign-indefinite measures, which can be solved using a suitable generalization of the well-known Sinkhorn algorithm (more properly, Schr\"odinger-Fortet-Sinkhorn algorithm, since it was explicit in the work of Fortet \cite{essid2019traversing} on the Schr\"odinger bridge problem). The generalization consists in modifying the iteration steps by scaling with values computed as the positive roots of a quadratic polynomial -- the classical Sinkhorn ``diagonal-scaling'' iteration applies when the adjacency matrix is sign-definite (with entries $\geq 0$) \cite{chen2015optimal, georgiou2015positive}. The proposed method was applied to determine the promotion and inhibition interactions in gene regulatory networks \cite{dong2023negative}.

Whereas determining the sign (promotion/inhibition) in multi-subsystem/multi-species relations may be tangible, determining the precise strength may be challenging as it is not typically accessible in complex and large-scale networks. For instance, in gene regulatory networks, protein expression reveals marginal information on nodes, and statistics on increased/decreased expression provide a sign for the promotion/inhibitive coupling between nodes. However, the strength of such pair-wise or, more generally, multi-gene interactions cannot be directly observed due to stupendously complicated effects across the network. Yet it is of at most importance in inferring potential fitness advantages on organisms by their combined effect \cite{tannenbaum2004solution}. Then, in ecological networks, quantifying interspecies interactions can also be challenging due to the intricate complexity and variability of the overall activity across the network, including environmental factors, habitat structures, and species behaviors \cite{bastolla2009architecture, stouffer2011compartmentalization}. Such scenarios motivate our problem, where partial information at the level of nodes can be assumed reliable, whereas the strength (but not the sign) of respective interactions is less so. Thereby, we consider directional marginal-wise sign templates that may encode the combined effect of interactions at the level of nodes, and assume nodal data as given, from which we seek to identify strength levels of pair-wise or multi-node interactions. 

Evidently, in real-world systems that may be impacted by higher-order interactions, describing these solely using networks may lead to a loss of higher-order information \cite{wolf2016advantages}. A higher-order network (or a hypergraph) is a generalization of a network, in which its hyperedges can connect any number of nodes \cite{berge1984hypergraphs}. Examples of higher-order networks encompass various domains, including email communication networks, metabolic networks, and protein-protein interaction networks \cite{chen2023teasing,chen2023survey,newman2018networks}. Computational methods and the influx of concepts such as entropy \cite{chen2020tensor}, controllability \cite{chen2021controllability}, and various similarity measures \cite{surana2022hypergraph} are rapidly entering the toolbox for dealing with higher-order networks. The contribution in the present paper can be viewed from this angle, formulating a basic problem of inference of multi-node interactions based on observed partial information and priors.

Thus, in this article, we formulate the problem and propose a framework aimed at tackling the inverse problem of estimating edge weights in a network by minimizing the Kullback–Leibler divergence between a prior and a posterior concerning directional marginal-wise sign patterns. Similar to our previous work \cite{dong2023negative}, this problem can be solved using a Sinkhorn-like algorithm. This algorithm amounts to implementing coordinate ascent to maximize a concave functional, leading to linear convergence rate. The algorithm represents a significant generalization of the standard solver of the entropic regularized optimal transport problem and the Schr\"odinger bridges problem, which has attracted increasing attention among the fields of control~\cite{chen2016relation,chen2021optimal}, machine learning~\cite{chen2021optimal}, and image processing~\cite{peyre2019computational}. Furthermore, is applies to higher-order networks, where we determine hyperedge weights based on directional marginal-wise sign-indefinite adjacency tensors (which are generalizations of adjacency matrices).

The organization of the remaining article is as follows. In Section \ref{sec:prelim}, we provide necessary background knowledge of signed networks and the Sinkhorn algorithm. Section \ref{sec:sinkhorn} presents the generalized Sinkhorn algorithm along with established theoretical results for our proposed problem. Section \ref{sec:general} discusses the generalization to general higher-order networks with the proposed method. Section \ref{sec:num} includes three numerical examples to validate the algorithm. Finally, we conclude with future directions in Section \ref{sec:conclude}.

\section{Preliminaries}\label{sec:prelim}
\subsection{Signed Networks}
Signed networks refer to network structures where the relationships between nodes are sign-indefinite ($\textbf{A}_{ij}\gtreqqless 0$) \cite{zaslavsky1982signed}. These networks capture not only affinity between nodes but also the nature of respective interactions, whether these represent cooperation, friendship, antagonism, rivalry, or other types of expression. Signed networks are prevalent across diverse domains, spanning sociology \cite{yuan2017edge}, ecology \cite{dale2021quantitative}, and signal processing \cite{dittrich2020signal}. Mathematically, a signed network can be represented by a sign-indefinite adjacency matrix, where its entries take values of $-1$, $0$, and $1$, or by a suitably weighted counterpart. 

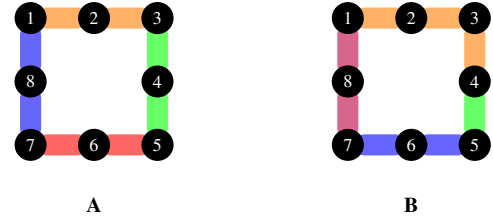
\begin{figure}[t]
\centering
\scalebox{0.8}{
\begin{tikzpicture}[scale=1]
\node[vertex,text=white,scale=0.85] (v1) {1};
\node[vertex,right of=v1,text=white,scale=0.85, node distance=30pt] (v2) {2};
\node[vertex,right of=v2,text=white,scale=0.85, node distance=30pt] (v3) {3};
\node[vertex,below of=v3,text=white,scale=0.85, node distance=30pt] (v4) {4};
\node[vertex,below of=v4,text=white,scale=0.85, node distance=30pt] (v5) {5};
\node[vertex,left of=v5,text=white,scale=0.85, node distance=30pt] (v6) {6};
\node[vertex,left of=v6,text=white,scale=0.85, node distance=30pt] (v7) {7};
\node[vertex,above of=v7,text=white,scale=0.85, node distance=30pt] (v8) {8};
\node[below of=v6,yshift=0cm]  (A) {\textbf{A}};

\node[vertex,right of=v1,text=white,scale=0.85, node distance=150pt] (v11) {1};
\node[vertex,right of=v11,text=white,scale=0.85, node distance=30pt] (v21) {2};
\node[vertex,right of=v21,text=white,scale=0.85, node distance=30pt] (v31) {3};
\node[vertex,below of=v31,text=white,scale=0.85, node distance=30pt] (v41) {4};
\node[vertex,below of=v41,text=white,scale=0.85, node distance=30pt] (v51) {5};
\node[vertex,left of=v51,text=white,scale=0.85, node distance=30pt] (v61) {6};
\node[vertex,left of=v61,text=white,scale=0.85, node distance=30pt] (v71) {7};
\node[vertex,above of=v71,text=white,scale=0.85, node distance=30pt] (v81) {8};
\node[below of=v61,yshift=0cm]  (B) {\textbf{B}};

\begin{pgfonlayer}{background}
\begin{scope}[transparency group,opacity=1]
\draw[edge,color=orange] (v1) -- (v2) -- (v3);
\draw[edge,color=green] (v3) -- (v4) -- (v5);
\draw[edge,color=red] (v5) -- (v6) -- (v7);
\draw[edge,color=blue] (v7) -- (v8) -- (v1);

\draw[edge,color=orange] (v11) -- (v21)-- (v31) -- (v41);
\draw[edge,color=green] (v41) -- (v51);
\draw[edge,color=blue] (v51) -- (v61) -- (v71);
\draw[edge,color=purple] (v11) -- (v81) -- (v71);
\end{scope}
\end{pgfonlayer}
\end{tikzpicture}}
\caption{Every hyperedge in the example is colored by one color. {\bf (A)} 3-uniform higher-order network with hyperedges $e_1=\{1,2,3\}$, $e_2=\{3,4,5\}$, $e_3=\{5,6,7\}$ and $e_4=\{8,7,1\}$. {\bf (B)} Non-uniform higher-order network with hyperedges $e_1=\{1,2,3,4\}$, $e_2=\{4,5\}$, $e_3 = \{5,6,7\}$ and $e_3=\{1,8,7\}$. }
\label{fig:hypergraph}
\end{figure}

Signed higher-order networks generalize signed networks by representing relationships among multiple entities or nodes with positive or negative signs. 
When all hyperedges link the exact same $k$ number of nodes, the higher-order network is termed $k$-uniform (Figure \ref{fig:hypergraph}). Every $k$-uniform higher-order network with $n$ nodes can be represented by a $k$th-order $n$-dimensional adjacency tensor (i.e., a multidimensional array generalized from matrices and vectors) \cite{chen2020tensor, chen2021controllability}. Thus, similar to signed networks, a signed $k$-uniform higher-order network can be characterized by a sign-indefinite adjacency tensor $\mathcal{A}\in\mathbb{R}^{n\times n\times \stackrel{k}{\cdots}\times n}$ defined as
\begin{equation}\label{eq:tensor-adj}
    \mathcal{A}_{i_1i_2\cdots i_k} = 
    \begin{cases}
       \phantom{-}1  &\text{ if } e \text{ forms a positive interaction}\\
       -1  &\text{ if } e \text{ forms a negative interaction}\\
       \phantom{-}0  &\text{ otherwise}
    \end{cases},
\end{equation}
where $e$ denotes the hyperedge containing node $i_1, i_2, \dots, i_k$.

\subsection{Schr\"odinger Bridges Problem}
The Schrödinger Bridges Problem (SBP) originated as a problem in large deviations theory, as the problem to minimize the Kullback-Leibler divergence, a measure of dissimilarity between probability distributions \cite{chen2016relation,chen2021stochastic}, between prior and posterior, so that the latter agrees with estimated statistics. 
Thus, if $\textbf{Q}\in\mathbb{R}_{\geq 0}^{n\times n}$ with non-negative entries represents a ``prior'' joint probability between two random variables, and if the corresponding to positive probability vectors $\textbf{p}\in\mathbb{R}^n_{> 0}$ and $\textbf{q}\in\mathbb{R}^n_{> 0}$ are not consistent with $\textbf{Q}$, then SBP seeks to adjust the values in $\textbf{Q}$ so as to derive a ``posterior'' $\textbf{P}\in\mathbb{R}^{n\times n}_{\geq 0}$ that is consistent with the specified marginals. The posterior matrix $\textbf{P} \in \mathscr{C}$ is defined as the set of admissible elements from a feasible set $\mathscr{C} = \mathscr{C}_0 \cap \mathscr{C}_1 \cap \mathscr{C}_2$, where $\mathscr{C}_0 := \{\textbf{P}~|~ \textbf{P}_{ij} \geq 0\}$, $\mathscr{C}_1 := \{\textbf{P}~|~ \sum_{j = 1}^n \textbf{P} = \textbf{p}\}$, and $\mathscr{C}_2 := \{\textbf{P}~|~ \sum_{i = 1}^n \textbf{P} = \textbf{q}\}$.

\begin{algorithm}[htb!]
\caption{Sinkhorn algorithm}\label{alg:Sinkhorn}
\begin{algorithmic}[1]
\STATE Initialize $\boldsymbol{\nu}\in\mathbb R^n$, e.g., setting $\boldsymbol{\nu}=\textbf{0}$.\\[0.045in]
\STATE 
For $j=1,2,\dots,n$, determine 
\begin{equation*}
\mu_j= \log(p_j/\sum_{i=1}^n \textbf{Q}_{ij}\exp(-\nu_i)). 
\end{equation*}
\STATE 
For $i=1,2,\dots, n$, determine 
\begin{equation*}
\nu_i= \log(q_i/\sum_{j=1}^n \textbf{Q}_{ij}\exp(-\mu_j)).
\end{equation*}
\STATE Repeat steps 2 and 3 until convergence.
\end{algorithmic}
\end{algorithm}

The standard SBP seeks $\textbf{P}$ as the minimizer of the Kullback-Leibler divergence to the prior $\textbf{Q}$, i.e., as the solution to
\begin{align}
   \textbf{P}^\star = \argmin_{\textbf{P}\in \mathscr{C} } \mbox{KL}(\textbf{P}|\textbf{Q})= \argmin_{\textbf{P}\in \mathscr{C} }  \sum_{i=1}^n\sum_{j=1}^n \textbf{P}_{ij}\log{(\frac{\textbf{P}_{ij}}{\textbf{Q}_{ij}})}. 
\end{align}
Throughout we follow the convention that $\log(0/0) = 0$. The objective function is strictly convex, and once the feasible set $\mathscr{C}$ is non-empty, an optimizer $\textbf{P}^*$ always exists. Additionally, a closed-form solution can be derived from the first-order optimality condition by employing the Lagrangian, i.e., 
\begin{align*}
\calL(\textbf{P},\boldsymbol{\mu},\boldsymbol{\nu}) :=
\! \! \! &\sum_{ij|\textbf{Q}_{ij}\neq 0}^n \!\!\! \textbf{P}_{ij} \log (\frac{\textbf{P}_{ij}}{\textbf{Q}_{ij}})+ \sum_{j=1}^n \mu_{j} \big(\sum_{i=1}^n \textbf{P}_{ij}-q_{j}\big)\\ 
&+  \sum_{i=1}^n \nu_{i} \big(\sum_{j=1}^n \textbf{P}_{ij} - p_{i} \big),
\end{align*}
where $\mu_j$ is the $j$th entry of $\boldsymbol{\mu}$ (similarly for $q_j$, $\nu_i$, and $p_i$). Therefore, 
\begin{align}\label{eq:Pstar}
    \textbf{P}_{ij}^{*} =  \textbf{Q}_{ij}  \exp{(-\nu_{i}-\mu_{j})}
\end{align}
with $\boldsymbol{\mu}$ and $\boldsymbol{\nu}$ are the dual variables for the marginal constraints
\begin{align}\label{eq:constraint1}
\sum_{j=1}^n \textbf{P}_{ij}^\star = \textbf{p} \text{ and } \ \sum_{i=1}^n \textbf{P}_{ij}^\star = \textbf{q}.
\end{align}
The solution in \eqref{eq:Pstar} has to satisfy the marginal constraints in $\mathscr{C}_{1}$ and $\mathscr{C}_2$, and the optimizer can be thus obtained by the Sinkhorn scheme that iteratively updates the two dual variables (Algorithm \ref{alg:Sinkhorn}). The linear convergence rate of the Sinkhorn algorithm can be established using the Hilbert metric, and interested readers are referred to \cite{peyre2019computational} for further details.

The SBP can be generalized to the multidimensional setting, taking the form
\begin{equation}\label{eq:tensor_entropy}
\mathcal{P}^\star = \argmin_{\mathcal{P}\in \mathscr{C} }  \sum_{i_1 i_2 \cdots i_k} \mathcal{P}_{i_1i_2\cdots i_{k}}\left( \log{\Big(\frac{\mathcal{P}_{i_1i_2\cdots i_{k}}}{\mathcal{Q}_{i_1i_2\cdots i_{k}}}\Big)}\right),
\end{equation}
where $\mathcal{Q}\in\mathbb{R}^{n\times n\times \stackrel{k}{\cdots}\times n}_{\geq 0}$ represents the given prior and $\mathcal{P}\in\mathbb{R}^{n\times n\times \stackrel{k}{\cdots}\times n}_{\geq 0}$ denotes the posterior. The posterior $\mathcal{P}$ satisfies the marginal constraint $\mathcal{P} \in \mathscr{C}_0 \cap \mathscr{C}_{1} \cap \cdots \cap \mathscr{C}_{k}$ where each $\mathscr{C}_{\ell}$ is defined as 
$
\calC_{\ell} = \{\mathcal{P}~|~\sum_{i_1\cdots i_k/i_\ell} \mathcal{P} = \textbf{p}^{(\ell)}\} 
$
for every marginal vector $\textbf{p}^{(l)}\in\mathbb{R}^n_{>0}$. Similar to the standard Sinkhorn iteration in Algorithm \ref{alg:Sinkhorn}, a closed-form solution of the optimizer can be derived, obtained by a gradient ascent method with, similarly, a linear convergence rate \cite{luo1992convergence}.

\section{Sinkhorn with Directional Sign Templates}\label{sec:sinkhorn}
The underlying principle in our previous work~\cite{dong2023data,dong2023negative} assumes a predetermined prior, with a given sign template applied across all directions of the marginal constraints. However, such an assumption may not always be practical for real-world applications. In complex systems, precisely capturing relationships can be challenging due to uncertainties or the inability to directly measure certain edge signs. This scenario would lead to the creation of multiple distinct signed networks, each representing alternative interpretations of the system's dynamics. For instance, in a social network where friendships and animosities exist, the absence of direct observations for several relationships may lead to the construction of multiple signed social networks with differing assumptions about uncertain relationships. Since sign patterns may be subjective and their precise determination may be impossible under certain circumstances, a framework that allows for directional sign templates as an initial guess is necessary and critical.

We introduce a non-trivial generalization of the Sinkhorn algorithm, namely Sinkhorn with directional sign templates. We first start with the classical two-dimensional case by considering a prior matrix $\textbf{Q}\in\mathbb{R}^{n\times n}_{\geq 0}$ and two specified positive marginal vectors $\textbf{p}\in\mathbb{R}^n_{>0}$ and $\textbf{q}\in\mathbb{R}^n_{>0}$. In contrast to the previous work, the problem now admits two directional marginal-wise sign patterns. The constraints for the posterior matrix $\textbf{P}\in\mathbb{R}^{n\times n}_{\geq 0}$ in \eqref{eq:constraint1} are thus generalized to 
\begin{equation}\label{eq:constraint2}
    \begin{aligned}
    \mathscr{C}^\prime_1 &:= \{\bfP~|~ \sum_{j=1}^n [\textbf{X} * \bfP]_{ij} = p_j\},\\
    \mathscr{C}^\prime_2 &:= \{\bfP~|~ \sum_{i=1}^n [\textbf{Y} * \bfP]_{ij} = q_i\},
\end{aligned}
\end{equation}
where $\textbf{X}\in\mathbb{R}^{n\times n}$ and $\textbf{Y}\in\mathbb{R}^{n\times n}$ are two distinct sign-indefinite adjacency matrices of the network, and $*$ denotes the element-wise multiplication.

\begin{problem}\label{prob:2-dim}
The SBP seeks the most likely posterior $\textbf{P}\in\mathbb{R}^{n\times n}_{\geq 0}$ with respect to a given prior $\textbf{Q}\in\mathbb{R}^{n\times n}_{\geq 0}$, achieved by minimizing the relative entropy
\begin{align}\label{prob:1}
    \textbf{P}^\star = \argmin_{\textbf{P}} \sum_{i=1}^n\sum_{j=1}^n \textbf{P}_{ij} \left( \log \left(\frac{\textbf{P}_{ij}}{\textbf{Q}_{ij}}\right) -1 \right)
\end{align}
with the convention $\log{(0/0)} = 0$. We intentionally add $-1$ at the end for simplicity in later computations. The posterior $\textbf{P}$ is also required to meet the marginal constraints
\begin{equation}\label{eq:constraint3}
    \sum_{i=1}^n \textbf{X}_{ij}\textbf{P}_{ij} = p_{j}\text{  and }  \sum_{j=1}^n \textbf{Y}_{ij}\textbf{P}_{ij} = q_{i}.
\end{equation}
\end{problem}

Our problem considers two distinct sign patterns, indicated by $\textbf{X}$ and $\textbf{Y}$, for different marginals while maintaining a convex objective. 
First, we demonstrate that the optimizer has a closed-form solution under the feasibility assumption.

\begin{assumption}[Existence]\label{assum:1}
The solution to Problem \ref{prob:2-dim} exists. In other words, the set of possible solutions is defined as
\begin{align*}
      \{\textbf{P}~|~\textbf{P}\in \mathscr{C}^\prime_1 \cap \mathscr{C}^\prime_2 \} \cap \{\textbf{P}~|~KL(\textbf{P}~|~\textbf{Q}) < +\infty\}
\end{align*}
is non-empty.
\end{assumption}

Note that a simple feasibility test may not exist for the problem. It is well-known that certain moment problems lack a straightforward feasibility test, as discussed in~\cite{georgiou2006relative}.

\begin{proposition}[Closed-form solution]
The optimizer \textbf{P} has a closed-form solution computed as 
\begin{align}
    \textbf{P}^\star_{ij} =
\begin{cases}
    \textbf{Q}_{ij} \exp{(-\mu_j \textbf{X}_{ij} - \nu_i \textbf{Y}_{ij})} &\text{ for } \textbf{X}_{ij}\neq 0,~\textbf{Y}_{ij} \neq 0\\
    0 &\mbox{ for } \textbf{X}_{ij} = 0, ~\textbf{Y}_{ij}=0
\end{cases},
\end{align}
where $\boldsymbol{\mu}\in\mathbb R^{n}$ and $\boldsymbol{\nu}\in\mathbb R^{n}$ are the Lagrangian multipliers of the constraints in \eqref{eq:constraint2}.
\end{proposition}

\begin{proof}
We first rewrite the Lagrangian $\calL$ as
\begin{equation}\label{eq:Lagrangian}
\begin{split}
{\mathcal L}(\mathbf{\bfP},\boldsymbol{\mu},\boldsymbol{\nu})&:=
\! \! \! \sum_{ij|\bfQ_{ij}\neq 0} \!\!\! \bfP_{ij} \left(\log (\frac{\bfP_{ij}}{\bfQ_{ij}})-1\right)\\ 
&+ 
\sum_{j} \mu_j \big(\sum_{i=1}^n  \textbf{X}_{ij} \bfP_{ij}-p_j\big)\\&+  
\sum_{i} \nu_i \big(\sum_{j=1}^n \textbf{Y}_{ij} \bfP_{ij} - q_i \big).
\end{split}
\end{equation}
The first-order optimality condition $\partial \mathcal L/\partial \bfP_{ij}=0$ yields
\begin{align*}
   \log(\frac{\bfP_{ij}}{\bfQ_{ij}}) + \mu_j \textbf{X}_{ij} + \nu_i \textbf{Y}_{ij}=0,
\end{align*}
for $i,j$  with $\bfQ_{ij}>0$. Otherwise, $\bfP_{ij}=0$. 
Therefore, the optimizer must have the following functional dependence on the Lagrange multipliers, i.e., 
\begin{align}\label{eq:Pistar}
 \bfP^*_{ij} = \bfQ_{ij}\exp\big(-\mu_j \textbf{X}_{ij}-\nu_i\textbf{Y}_{ij}\big),
\end{align}
if $\bfQ_{ij} \neq 0$, and $\bfP^{*}_{ij} = 0$ otherwise.
\end{proof}

The optimal kernel can be rewritten in a more specific form 
\begin{align}\label{eq:P_opt}
    \bfP^*_{ij} = 
\begin{cases}
    \bfQ_{ij}\exp\big(- \mu_j - \nu_i\big),  &\mbox{if } \textbf{X}_{ij}>0,~\textbf{Y}_{ij}>0,\\
    \bfQ_{ij}\exp\big(- \mu_j + \nu_i\big),  &\mbox{if } \textbf{X}_{ij}>0,~\textbf{Y}_{ij}<0,\\
    \bfQ_{ij}\exp\big(+ \mu_j - \nu_i\big),  &\mbox{if } \textbf{X}_{ij}<0,~\textbf{Y}_{ij}>0,\\
    \bfQ_{ij}\exp\big(+ \mu_j + \nu_i\big),  &\mbox{if } \textbf{X}_{ij}<0,~\textbf{Y}_{ij}<0,\\
    0,                                  &\mbox{if } \textbf{X}_{ij}=0,~\textbf{Y}_{ij}=0.
\end{cases}
\end{align}
To write \eqref{eq:P_opt} in a more compact form, we define  indicator matrices $\bfQ^{++}$, $\bfQ^{+-}$, $\bfQ^{-+}$, and $\bfQ^{--}$ as 
\begin{align*}
    \bfQ^{++} &:= \{ \bfQ_{ij}~|~\textbf{X}_{ij}>0,~\textbf{Y}_{ij}>0\},\\
    \bfQ^{+-} &:= \{ \bfQ_{ij}~|~\textbf{X}_{ij}>0,~\textbf{Y}_{ij}<0\},\\
    \bfQ^{-+} &:= \{ \bfQ_{ij}~|~\textbf{X}_{ij}<0,~\textbf{Y}_{ij}>0\},\\
    \bfQ^{--} &:= \{ \bfQ_{ij}~|~\textbf{X}_{ij}<0,~\textbf{Y}_{ij}<0\}.
\end{align*}
It is straightforward to see that 
\begin{align*}
    \bfQ = \bfQ^{++}+\bfQ^{+-}+\bfQ^{-+}+\bfQ^{--}.
\end{align*}
We also define variables $\alpha_j = \exp(\mu_j)$ and $\beta_i = \exp(\nu_i)$.  The optimizer  \eqref{eq:P_opt} thus can be rewritten as 
\begin{equation}
\begin{split}
\bfP_{ij}^\star = \bfQ_{ij}^{++}\alpha_j\beta_i + \bfQ_{ij}^{+-}\alpha_j\beta^{-1}_i +\bfQ_{ij}^{-+}\alpha^{-1}_j\beta_i+ \bfQ_{ij}^{--}\alpha^{-1}_j\beta^{-1}_i,
\end{split}
\end{equation}
which satisfies the constraints in \eqref{eq:constraint2}, i.e., 
\begin{align*}
\sum_{i=1}^n &(\bfQ_{ij}^{++}\beta_i + \bfQ_{ij}^{+-}\beta^{-1}_i)\alpha_j \\ 
&\phantom{xxxxxxxxxxxxx}+(\bfQ_{ij}^{-+}\beta_i + \bfQ_{ij}^{--}\beta^{-1}_i)\alpha^{-1}_j = p_{j},\\ 
\sum_{j=1}^n &(\bfQ_{ij}^{++}\alpha_j+\bfQ_{ij}^{-+}\alpha^{-1}_j)\beta_i\\ 
&\phantom{xxxxxxxxxxxxx}+ (\bfQ_{ij}^{+-}\alpha_j + \bfQ_{ij}^{--}\alpha^{-1}_j)\beta^{-1}_i = q_{i}.   
\end{align*}
Additionally, to simplify the constraints, we define 
\begin{align*}
&a_j^{\{\bfQ++,\bfQ+-\}}(\boldsymbol{\beta}):= \sum_{i=1}^n \left(\bfQ_{ij}^{++}\beta_i + \bfQ_{ij}^{+-}\beta^{-1}_i\right),\\
&b_j^{\{\bfQ-+,\bfQ--\}}(\boldsymbol{\beta}):= \sum_{i=1}^n \left(\bfQ_{ij}^{-+}\beta_i+ \bfQ_{ij}^{--}\beta^{-1}_i\right),\\
&a_i^{\{\bfQ++,\bfQ-+\}}(\boldsymbol{\alpha}):= \sum_{j=1}^n \left(\bfQ_{ij}^{++}\alpha_j+\bfQ_{ij}^{-+}\alpha^{-1}_j\right),\\
&b_i^{\{\bfQ+-,\bfQ--\}}(\boldsymbol{\alpha}):= \sum_{j=1}^n \left(\bfQ_{ij}^{+-}\alpha_j + \bfQ_{ij}^{--}\alpha^{-1}_j\right).
\end{align*} 
We can then rewrite the constraints as 
\begin{subequations}\label{eq:update-constraint}
\begin{align}
p_j &= a_j^{\{\bfQ++,\bfQ+-\}}(\boldsymbol{\beta}) \alpha_j - b_j^{\{\bfQ-+,\bfQ--\}}(\boldsymbol{\beta}) \alpha^{-1}_j,\\
\label{eq:B_i}
q_i &= a_i^{\{\bfQ++,\bfQ-+\}}(\boldsymbol{\alpha}) \beta_i - b_i^{\{\bfQ+-,\bfQ--\}}(\boldsymbol{\alpha}) \beta^{-1}_i.
\end{align} 
\end{subequations}


We observe that $\alpha_j$ can be explicitly computed from the marginal constraint $p_j$, while keeping the vector $\boldsymbol{\beta}$ fixed. Similarly, $\beta_i$ can be computed from $q_i$. To see this, we can consider the following function 
\begin{align*}
    f(x)=ax+bx^{-1},
\end{align*}
where $a$ and $b$ are both positive. Thus, for a given value $c$, a solution to $f(x)=c$ can be readily obtained as the positive root of a quadratic equation, computed as
\begin{subequations}\label{eq:logquad}
\begin{align}\label{eq:logquad1}
    x= \left(\frac{-c+\sqrt{c^2 + 4ab}}{2b}\right)=:g(a,b,c).
\end{align}
For our purposes, the case where $b=0$ is also of interest, i.e.,
    \begin{align}\label{eq:logquad2}
    x= \left(\frac{a}{c}\right)=:g(a,0,c).
\end{align}
\end{subequations}
The proposed generalized Sinkhorn algorithm for two directional sign templates is summarized in Algorithm \ref{alg:sinkhorn-like}.

\begin{algorithm}[htb!]
\caption{Generalized Sinkhorn algorithm}\label{alg:sinkhorn-like}
\begin{algorithmic}[1]
\STATE Initialize $\boldsymbol{\beta}\in\mathbb R^n$, e.g., setting $\boldsymbol{\nu}=\textbf{1}$.\\[0.045in]
\STATE Compute $\bfQ^{++}$, $\bfQ^{+-}$, $\bfQ^{-+}$, and $\bfQ^{--}$.\\[0.045in]
\STATE 
For $j=1,2,\dots,n$, determine 
\begin{subequations}\label{eq:updates}
\begin{align}\label{eq:mu-update}
  \alpha_j=g\left(a_j^{\{\bfQ++,\bfQ+-\}}(\boldsymbol{\beta}),b_j^{\{\bfQ-+,\bfQ--\}}(\boldsymbol{\beta}),p_j\right).
\end{align}
\STATE 
For $i=1,2,\dots, n$, determine 
\begin{align}\label{eq:nu-update}
  \beta_i=g\left(a_i^{\{\bfQ++,\bfQ-+\}}(\boldsymbol{\alpha}),b_i^{\{\bfQ+-,\bfQ--\}}(\boldsymbol{\alpha}),q_i\right).
\end{align}
\end{subequations}
\STATE Repeat steps 2 and 3 until convergence.
\STATE Compute the optimizer $\bfP^\star$.
\end{algorithmic}
\end{algorithm}

\begin{proposition}[Convergence]\label{thm:convergence}
Algorithm \ref{alg:sinkhorn-like} converges to its optimal with a linear convergence rate under Assumption \ref{assum:1}.   
\end{proposition}

\begin{proof}
The convergence of the algorithm can be understood through the strong duality of the problem. We first substitute the optimizer \eqref{eq:Pistar} into the Lagrangian \eqref{eq:Lagrangian}, obtaining
\begin{align}
    h(\boldsymbol{\mu},\boldsymbol{\nu}) = \sum_{i=1}^n\sum_{j=1}^n -\bfP^\star_{ij} - \mu_jp_j - \nu_iq_i.
\end{align}
The dual problem is therefore a maximization problem with a strictly concave objective defined as
\begin{align}\label{prob:dual}
    \argmax_{\boldsymbol{\mu},\boldsymbol{\nu}} \ \ h(\boldsymbol{\mu},\boldsymbol{\nu}),
\end{align}
where the Slater's condition is satisfied, ensuring that strong duality holds.

Given that the objective is strictly convex, we can obtain the optimizer $\mathbf{P}^\star$ using a coordinate ascent strategy, while ensuring strong duality holds. Specifically, the coordinate-wise optimization of the dual~\eqref{prob:dual}, i.e., 
\begin{align*}
    \mu_j = \argmax_{\mu_j}\ h(\boldsymbol{\mu},\boldsymbol{\nu}) \text{ and }
    \nu_i = \argmax_{\nu_i}\ h(\boldsymbol{\mu},\boldsymbol{\nu}),
\end{align*}
gives that
\begin{align*}
&\frac{\partial h(\mu_j,\nu_i)}{\partial \mu_j} = \sum_{i=1}^n \bfP_{ij}^\star  - p_{j}=0\\
&\frac{\partial h(\mu_j,\nu_i)}{\partial \nu_i} = \sum_{j=1}^n \bfP_{ij}^\star  - q_{i}=0
\end{align*}
at each step, leading to an update for $\mu_{j}$ and $\nu_i$ to satisfy \eqref{eq:update-constraint} and \eqref{eq:updates}. Therefore, the generalized Sinkhorn iteration in Algorithm~\ref{alg:sinkhorn-like} inherits the linear convergence rate of coordinate ascent~\cite[Theorem 2.1]{luo1992convergence}, as in the standard Sinkhorn. 
\end{proof}

Last but not least, the problem and obtained results (including Algorithm \ref{alg:sinkhorn-like} and Proposition \ref{thm:convergence}) can be extended to the multi-marginal setting for higher-order networks, with a predetermined $k$th-order $n$-dimensional prior $\mathcal{Q} \in \mathbb R^{n\times n\times \stackrel{k}{\cdots}\times n}_{\geq 0}$ and $k$ positive marginals 
$
\textbf{p}^{(1)}, \textbf{p}^{(2)},\dots, \textbf{p}^{(k)}\in\mathbb{R}^{n}_{>0}.
$ Moreover, every marginal constraint set for the posterior $\mathcal{P} \in \mathbb R^{n\times n\times \stackrel{k}{\cdots}\times n}_{\geq 0}$ is then defined as 
\begin{align*}
\mathscr{C}''_\ell:= \{ \mathcal{P}~|~ \sum_{i_1i_2\cdots i_k/i_\ell} \mathcal{X}^{(\ell)}_{i_1i_2\cdots i_k} \mathcal{P}_{i_1i_2\cdots i_k} = p^{(\ell)}_{i_\ell}\},
\end{align*}
where $\mathcal{X}^{(l)}\in \mathbb R^{n\times n\times \stackrel{k}{\cdots}\times n}$ are directional marginal-wise sign-indefinite adjacency tensors. Thus, all the admissible solution $\mathcal{P}$ has to be in the set
\begin{align*}
    \mathscr{C}'' = \mathscr{C}''_1 \cap \mathscr{C}''_2 \cap  \dots \cap \mathscr{C}''_k.
\end{align*}

\begin{problem}\label{prob:multi-dim}
The SBP seeks the most likely posterior $\mathcal{P} \in \mathbb R^{n\times n\times \stackrel{k}{\cdots}\times n}_{\geq 0}$ with respect to a given prior $\mathcal{Q} \in \mathbb R^{n\times n\times \stackrel{k}{\cdots}\times n}_{\geq 0}$, achieved by minimizing the relative entropy \eqref{eq:tensor_entropy}
over admissible solutions $\mathcal{P} \in \mathscr{C}''$.
\end{problem}

The problem is well-defined when it is feasible, and its solution can be obtained similarly to the two-dimensional case (Algorithm \ref{alg:sinkhorn-like}). Concerning convergence, the argument is analogous to our discussion in Proposition~\ref{thm:convergence}, as it relies on the gradient ascent nature of the problem~\cite{luo1992convergence}.

\section{General Higher-order Networks}\label{sec:general}
Every $k$-uniform higher-order network, where hyperedges have the same cardinality $k$, is known to possess a $k$th-order tensor representation \eqref{eq:tensor-adj}. However, representing general (or non-uniform) higher-order networks is not straightforward due to the potential variance in hyperedge cardinality. Nevertheless, a tensor representation can be achieved by converting the original non-uniform higher-order network to a uniform one through the addition of `virtual' nodes. These virtual nodes carry no nodal information but merely serve as placeholders to ensure every hyperedge has the same cardinality. Consequently, the converted uniform higher-order network features $k^{\max}$ cardinality for every hyperedge, with $k^{\max}$ representing the largest hyperedge cardinality in the original non-uniform higher-order network. A simple example of such conversion is illustrated in Figure~\ref{fig:uni-hypergraph}. Ultimately, the corresponding marginal distribution maintains the same number of nodes, with nodal information on every virtual node set to zero. Hence, this problem can be adapted to Problem \ref{prob:multi-dim}.
\begin{figure}[htb!]
\centering
\scalebox{0.8}{
\begin{tikzpicture}[scale=1]
\node[vertex,text=white,scale=0.85] (v1) {1};
\node[vertex,below of=v1,text=white,scale=0.85, node distance=30pt] (v2) {2};
\node[vertex,below of=v2,text=white,scale=0.85, node distance=30pt] (v3) {3};
\node[vertex,right of=v2,text=white,scale=0.85, node distance=60pt] (v4) {4};
\node[below of=v6,yshift=0cm]  (A) {\textbf{A}};

\node[vertex,right of=v1,text=white,scale=0.85, node distance=150pt] (v11) {1};
\node[vertex,below of=v11,text=white,scale=0.85, node distance=30pt] (v21) {2};
\node[vertex,below of=v21,text=white,scale=0.85, node distance=30pt] (v31) {3};
\node[vertex,right of=v21,text=white,scale=0.85, node distance=30pt] (v51) {5};
\node[vertex,right of=v51,text=white,scale=0.85, node distance=30pt] (v41) {4};
\node[below of=v61,yshift=0cm]  (B) {\textbf{B}};

\begin{pgfonlayer}{background}
\begin{scope}[transparency group,opacity=1]
\draw[edge,color=orange] (v1) -- (v2) -- (v3);
\draw[edge,color=green] (v1) -- (v4);
\draw[edge,color=red] (v3) -- (v4);

\draw[edge,color=orange] (v11) -- (v21)-- (v31);
\draw[edge,color=green] (v11) -- (v51) -- (v41);
\draw[edge,color=red] (v31) -- (v51) -- (v41);
\end{scope}
\end{pgfonlayer}
\end{tikzpicture}}
\caption{Uniformity conversion. (\textbf{A}) A non-uniform higher-order network with hyperedges $e_1=\{1,2,3\}$, $e_2=\{1,4\}$, and $e_3=\{3,4\}$ is converted to a $3$-uniform higher-order network by adding virtual node $5$. (\textbf{B}) The resulting 3-uniform higher-order network with hyperedges $e_1=\{1,2,3\}$, $e_2=\{1,4,5\}$, and $e_3=\{3,4,5\}$.}
\label{fig:uni-hypergraph}
\end{figure}
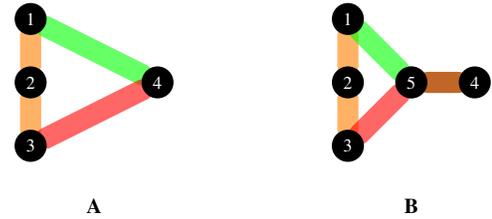

\section{Numerical Examples}\label{sec:num}
We conducted three numerical examples to illustrate our framework and the efficiency of the generalized Sinkhorn algorithm. The code for all experiments can be found at the following link: \url{https://github.com/dytroshut/marginal-wise-CDC}.
\begin{figure*}[htb!]
\centering
\begin{subfigure}[t]{0.5\textwidth}
\centering
\scalebox{0.9}{
\begin{tikzpicture}
[p/.style={circle,draw=black!45,fill=black!10,line width=0.5mm,inner sep=0pt,minimum size=6mm}]
\node (n1) at (0.6180,1.9021) [p] {$1$};
\node (n2) at (-0.6180,1.9021) [p] {$2$};
\node (n3) at (-1.6180,1.1756) [p] {$3$};			
\node (n4) at (-2,0) [p] {$4$};
\node (n5) at (-1.6180,-1.1756) [p] {$5$};
\node (n6) at (-.6180,-1.9021) [p] {$6$};
\node (n7) at (.6180,-1.9021) [p] {$7$};
\node (n8) at (1.6180,-1.1756) [p] {$8$};
\node (n9) at (2,0) [p] {$9$};
\node (n10) at (1.6180,1.1756) [p] {$10$};

\path[line width=.6mm][blue!90] (n1) edge [in=95,out=45,loop] node {} (n1);
\path[line width=.6mm][blue!90] (n5) edge [in=-180,out=-130,loop] node {} (n5);
\path[line width=.6mm][blue!90] (n7) edge [in=-40,out=-90,loop] node {} (n7);

\draw[line width=1mm,dotted][purple!90] (n1) to [bend left=15] node[midway,right](){} (n2);
\draw[line width=.6mm][blue!90] (n1) to [bend left=15] node[midway,right](){} (n3);
\draw[line width=.6mm][blue!90] (n1) to [bend left=15] node[midway,right](){} (n5);
\draw[line width=.6mm][blue!90] (n1) to [bend right=15] node[midway,right](){} (n7);
\draw[line width=.6mm][blue!90] (n1) to [bend right=15] node[midway,right](){} (n8);
\draw[line width=.6mm][blue!90] (n1) to [bend right=30] node[midway,right](){} (n9);

\draw[line width=.6mm][blue!90] (n2) to [bend left=20] node[midway,right](){} (n3);
\draw[line width=.6mm][blue!90] (n2) to [bend left=20] node[midway,right](){} (n7);
\draw[line width=.6mm][blue!90] (n2) to [bend right=30] node[midway,right](){} (n10);

\draw[line width=.6mm][blue!90] (n3) to [bend left=20] node[midway,right](){} (n5);
\draw[line width=.6mm][blue!90] (n3) to [bend left=15] node[midway,right](){} (n7);
\draw[line width=.6mm][blue!90] (n3) to [bend left=15] node[midway,right](){} (n8);

\draw[line width=.6mm][blue!90] (n4) to [bend left=20] node[midway,right](){} (n5);
\draw[line width=.6mm][blue!90] (n4) to [bend left=15] node[midway,right](){} (n8);

\draw[line width=.6mm][blue!90] (n5) to [bend left=15] node[midway,right](){} (n6);

\draw[line width=.6mm][blue!90] (n6) to [bend left=15] node[midway,right](){} (n7);
\draw[line width=.6mm][blue!90] (n6) to [bend left=15] node[midway,right](){} (n9);

\draw[line width=.6mm][blue!90] (n8) to [bend left=15] node[midway,right](){} (n9);
\draw[line width=.6mm][blue!90] (n8) to [bend left=15] node[midway,right](){} (n10);

\draw[line width=.6mm][blue!90] (n9) to [bend left=15] node[midway,right](){} (n10);
\end{tikzpicture}}
\end{subfigure}%
~ 
\begin{subfigure}[t]{0.5\textwidth}
\centering
\scalebox{0.9}{
\begin{tikzpicture}
[p/.style={circle,draw=black!45,fill=black!10,line width=0.5mm,inner sep=0pt,minimum size=6mm}]
\node (n1) at (0.6180,1.9021) [p] {$1$};
\node (n2) at (-0.6180,1.9021) [p] {$2$};
\node (n3) at (-1.6180,1.1756) [p] {$3$};			
\node (n4) at (-2,0) [p] {$4$};
\node (n5) at (-1.6180,-1.1756) [p] {$5$};
\node (n6) at (-.6180,-1.9021) [p] {$6$};
\node (n7) at (.6180,-1.9021) [p] {$7$};
\node (n8) at (1.6180,-1.1756) [p] {$8$};
\node (n9) at (2,0) [p] {$9$};
\node (n10) at (1.6180,1.1756) [p] {$10$};

\path[line width=.6mm][blue!90] (n1) edge [in=95,out=45,loop] node {} (n1);
\path[line width=.6mm][blue!90] (n5) edge [in=-180,out=-130,loop] node {} (n5);
\path[line width=.6mm][blue!90] (n7) edge [in=-40,out=-90,loop] node {} (n7);

\draw[line width=1mm,dotted][purple!90] (n1) to [bend left=15] node[midway,right](){} (n2);
\draw[line width=.6mm][blue!90] (n1) to [bend left=15] node[midway,right](){} (n3);
\draw[line width=.6mm][blue!90] (n1) to [bend left=15] node[midway,right](){} (n5);
\draw[line width=.6mm][blue!90] (n1) to [bend right=15] node[midway,right](){} (n7);
\draw[line width=.6mm][blue!90] (n1) to [bend right=15] node[midway,right](){} (n8);
\draw[line width=1mm,dotted][purple!90] (n1) to [bend right=30] node[midway,right](){} (n9);

\draw[line width=.6mm][blue!90] (n2) to [bend left=20] node[midway,right](){} (n3);
\draw[line width=.6mm][blue!90] (n2) to [bend left=20] node[midway,right](){} (n7);
\draw[line width=.6mm][blue!90] (n2) to [bend right=30] node[midway,right](){} (n10);

\draw[line width=.6mm][blue!90] (n3) to [bend left=20] node[midway,right](){} (n5);
\draw[line width=.6mm][blue!90] (n3) to [bend left=15] node[midway,right](){} (n7);
\draw[line width=.6mm][blue!90] (n3) to [bend left=15] node[midway,right](){} (n8);

\draw[line width=.6mm][blue!90] (n4) to [bend left=20] node[midway,right](){} (n5);
\draw[line width=.6mm][blue!90] (n4) to [bend left=15] node[midway,right](){} (n8);

\draw[line width=.6mm][blue!90] (n5) to [bend left=15] node[midway,right](){} (n6);

\draw[line width=.6mm][blue!90] (n6) to [bend left=15] node[midway,right](){} (n7);
\draw[line width=.6mm][blue!90] (n6) to [bend left=15] node[midway,right](){} (n9);

\draw[line width=.6mm][blue!90] (n8) to [bend left=15] node[midway,right](){} (n9);
\draw[line width=.6mm][blue!90] (n8) to [bend left=15] node[midway,right](){} (n10);

\draw[line width=1mm,dotted][purple!90] (n9) to [bend left=15] node[midway,right](){} (n10);
\end{tikzpicture}}
\end{subfigure}
\caption{The topology of a $10$-node network with sign templates $\textbf{X}$ (left) and $\textbf{Y}$ (right).}
\label{fig:topology}
\end{figure*}
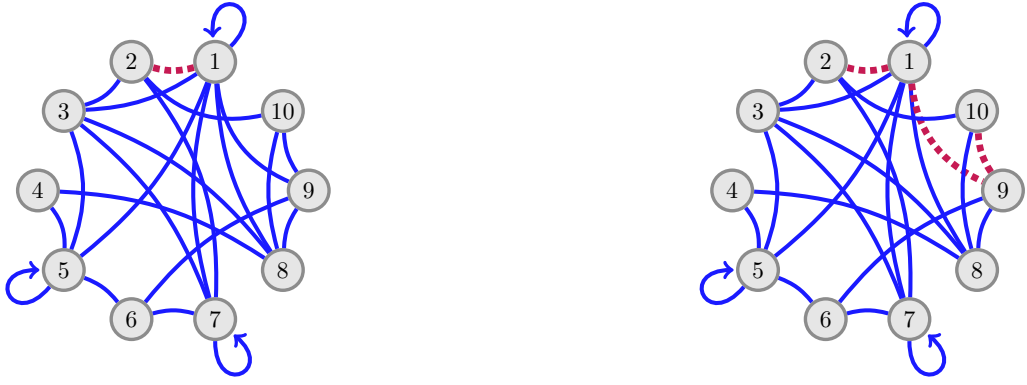

\subsection{Synthetic Example}
We first considered a simple synthetic example, where the sign-definite prior of the network is characterized by 
\begin{align*}
\bfQ = 
\begin{bmatrix}
    0 &1 &1 &1\\
    1 &0 &1 &1\\
    1 &1 &0 &1\\
    1 &1 &1 &0\\
\end{bmatrix}.
\end{align*}
The sign assigned on edges $\{1,2\}$ and $\{1,3\}$ is uncertain, with corresponding nodal information normalized as marginal distributions $\textbf{p} = [0.2, 0.3, 0.1, 0.4]$ and $\textbf{q} = [0.1, 0.1, 0.4, 0.4]$. The sign templates for the two templates are given by
\begin{align*}
&\textbf{X} = 
\begin{bmatrix}
    \phantom{-}0 &-1 &\phantom{-}1 &\phantom{-}1\\
    -1  &\phantom{-}0 &\phantom{-}1 &\phantom{-}1\\
    \phantom{-}1 &\phantom{-}1 &\phantom{-}0 &\phantom{-}1\\
    \phantom{-}1 &\phantom{-}1 &\phantom{-}1 &\phantom{-}0\\
\end{bmatrix}, \ 
\textbf{Y} = 
\begin{bmatrix}
    \phantom{-}0 &\phantom{-}1 &-1 &\phantom{-}1\\
    \phantom{-}1 &\phantom{-}0 &\phantom{-}1 &\phantom{-}1\\
    -1 &\phantom{-}1 &\phantom{-}0 &\phantom{-}1\\
    \phantom{-}1 &\phantom{-}1 &\phantom{-}1 &\phantom{-}0\\
\end{bmatrix}.
\end{align*}
According to Algorithm \ref{alg:sinkhorn-like}, the posterior $\bfP$ can be solved as
\begin{align*}
\bfP = 
\begin{bmatrix}
0    &0.0766    &0.1221    &0.1546\\
0.1269   &0    &0.1989    &0.2280\\
0.0815   &0.0011    &0    &0.0174\\
0.0546   &0.0223    &0.3231  &0    
\end{bmatrix}.
\end{align*}
Therefore, the posterior with sign-indefinite adjacencies thus has the form $\textbf{X} * \bfP$ and $\textbf{Y} * \bfP$.

\subsection{Ecological Network}
In ecological networks, interactions between species are often understood qualitatively, meaning that the signs of interactions (such as inhibition or promotion) can be determined with reasonable confidence, but quantifying the actual magnitudes may be challenging \cite{jeffries1974qualitative,logofet1982sign}. Under certain instances, even determining the signs of interactions can be obscure. Thus, it is necessary to consider multiple potential sign patterns when learning the magnitudes of these interactions. In this example, we considered an ecological network where ten species are represented by nodes, and their species-wise interactions are represented by edges. The abundance of each species is normalized and represented as marginal distributions given by
\begin{align*}
    \textbf{p} &= [0.1,0.05,0.05,0.15,0.2,0.05,0.03,0.07,0.25,0.05],\\
    \textbf{q} &= [0.05,0.1,0.05,0.2,0.07,0.15,0.05,0.25,0.03,0.05].
\end{align*}
The two marginal distributions originated from two potential sign templates are defined as
\begin{align*}
\small
\textbf{X} = 
\begin{bmatrix}
\phantom{-}1  &-1  &\phantom{-}1  &\phantom{-}0  &\phantom{-}1  &\phantom{-}0  &\phantom{-}1  &\phantom{-}0  &\phantom{-}1  &\phantom{-}1\\
-1  &\phantom{-}0  &\phantom{-}1  &\phantom{-}0  &\phantom{-}0  &\phantom{-}0  &\phantom{-}1  &\phantom{-}0  &\phantom{-}0  &\phantom{-}1\\
\phantom{-}1  &\phantom{-}1  &\phantom{-}0  &\phantom{-}0  &\phantom{-}1  &\phantom{-}0  &\phantom{-}1  &\phantom{-}1  &\phantom{-}0  &\phantom{-}0\\
\phantom{-}0  &\phantom{-}0  &\phantom{-}0  &\phantom{-}0  &\phantom{-}1  &\phantom{-}0  &\phantom{-}0  &\phantom{-}1  &\phantom{-}0  &\phantom{-}0\\
\phantom{-}1  &\phantom{-}0  &\phantom{-}1  &\phantom{-}1  &\phantom{-}1  &\phantom{-}1  &\phantom{-}0  &\phantom{-}0  &\phantom{-}0  &\phantom{-}0\\
\phantom{-}0  &\phantom{-}0  &\phantom{-}0  &\phantom{-}0  &\phantom{-}1  &\phantom{-}0  &\phantom{-}1  &\phantom{-}0  &\phantom{-}1  &\phantom{-}0\\
\phantom{-}1  &\phantom{-}1  &\phantom{-}1  &\phantom{-}0  &\phantom{-}0  &\phantom{-}1  &\phantom{-}1  &\phantom{-}0  &\phantom{-}0  &\phantom{-}0\\
\phantom{-}0  &\phantom{-}0  &\phantom{-}1  &\phantom{-}1  &\phantom{-}0  &\phantom{-}0  &\phantom{-}0  &\phantom{-}0  &\phantom{-}1  &\phantom{-}1\\
\phantom{-}1  &\phantom{-}0  &\phantom{-}0  &\phantom{-}0  &\phantom{-}0  &\phantom{-}1  &\phantom{-}0  &\phantom{-}1  &\phantom{-}0  &\phantom{-}1\\
\phantom{-}1  &\phantom{-}1  &\phantom{-}0  &\phantom{-}0  &\phantom{-}0  &\phantom{-}0  &\phantom{-}0  &\phantom{-}1  &\phantom{-}1  &\phantom{-}0
\end{bmatrix},
\end{align*}

\begin{align*}
\small
\textbf{Y} = 
\begin{bmatrix}
\phantom{-}1  &-1  &\phantom{-}1  &\phantom{-}0  &\phantom{-}1  &\phantom{-}0  &\phantom{-}1  &\phantom{-}0  &-1  &\phantom{-}1\\
-1  &\phantom{-}0  &\phantom{-}1  &\phantom{-}0  &\phantom{-}0  &\phantom{-}0  &\phantom{-}1  &\phantom{-}0  &\phantom{-}0  &\phantom{-}1\\
\phantom{-}1  &\phantom{-}1  &\phantom{-}0  &\phantom{-}0  &\phantom{-}1  &\phantom{-}0  &\phantom{-}1  &\phantom{-}1  &\phantom{-}0  &\phantom{-}0\\
\phantom{-}0  &\phantom{-}0  &\phantom{-}0  &\phantom{-}0  &\phantom{-}1  &\phantom{-}0  &\phantom{-}0  &\phantom{-}1  &\phantom{-}0  &\phantom{-}0\\
\phantom{-}1  &\phantom{-}0  &\phantom{-}1  &\phantom{-}1  &\phantom{-}1  &\phantom{-}1  &\phantom{-}0  &\phantom{-}0  &\phantom{-}0  &\phantom{-}0\\
\phantom{-}0  &\phantom{-}0  &\phantom{-}0  &\phantom{-}0  &\phantom{-}1  &\phantom{-}0  &\phantom{-}1  &\phantom{-}0  &\phantom{-}1  &\phantom{-}0\\
\phantom{-}1  &\phantom{-}1  &\phantom{-}1  &\phantom{-}0  &\phantom{-}0  &\phantom{-}1  &\phantom{-}1  &\phantom{-}0  &\phantom{-}0  &\phantom{-}0\\
\phantom{-}0  &\phantom{-}0  &\phantom{-}1  &\phantom{-}1  &\phantom{-}0  &\phantom{-}0  &\phantom{-}0  &\phantom{-}0  &\phantom{-}1  &\phantom{-}1\\
-1  &\phantom{-}0  &\phantom{-}0  &\phantom{-}0  &\phantom{-}0  &\phantom{-}1  &\phantom{-}0  &\phantom{-}1  &\phantom{-}0  &-1\\
\phantom{-}1  &\phantom{-}1  &\phantom{-}0  &\phantom{-}0  &\phantom{-}0  &\phantom{-}0  &\phantom{-}0  &\phantom{-}1  &-1  &\phantom{-}0
\end{bmatrix}.
\end{align*}
Specifically, the coefficient between Species 1 and 2 is known to be negative, while the interactions between Species 1 and 9, as well as between Species 9 and 10, are uncertain. The topologies of both of the 10-node signed ecological networks are depicted in Figure~\ref{fig:topology}.
The prior is set as $\textbf{Q} = |\textbf{X}| = |\textbf{Y}|$ for simplicity, without any initial guess of the edge weights. The posterior \eqref{eq:examplePi} obtained using Algorithm~\ref{alg:sinkhorn-like} represents the maximum likelihood estimation based on the observed information from the ecological networks.

\begin{figure*}[t]
\begin{equation}\label{eq:examplePi}
\bfP = 
\begin{bmatrix}
0.0108  &0.0990   &0.0817    & 0   & 0.0412   &0   &0.0167 &0   &0.0194   &0.0292\\
0.1497  &0   &0.1195   &  0   & 0   &0    &0.0291   &   0     & 0    &0.0511\\
0.0286  &0.0002   &0   &0    &0.0023    &0    &0.0009    &0.0180    &     0    &     0\\
0    &    0     &    0      &   0   & 0.0168    &     0    &     0    &0.1332    &     0   &    0\\
0.0010     &    0    &0.0063    &0.1439    &0.0038    &0.0450    &     0      &   0    & 0     &    0\\
0     &    0     &    0      &   0   & 0.0059    &     0    &0.0024     &    0    &0.0417    &     0\\
0.0006    &0.0002    &0.0035      &   0   &      0    &0.0249   & 0.0009     &    0    & 0     &    0\\
0     &    0    &0.0025    &0.0561   &      0     &    0    &     0     &    0    &0.0104    &0.0011\\
0.0850     &    0     &    0      &   0   &      0    &0.0801    &     0    &0.0536    & 0    &0.0314\\
0.0015    &0.0006     &    0      &   0   &      0    &     0    &     0    &0.0453    &0.0027    &     0 
\end{bmatrix}
\end{equation}
\hrulefill
\vspace*{4pt}
\end{figure*}

\subsection{Convergence}
To demonstrate the convergence rate of the proposed framework, we presented a multi-marginal case study involving tensors (Problem \ref{prob:multi-dim}), where both the prior and posterior are third-order tensors associated with a higher-order network. The marginal violations are defined as 
\begin{equation*}
    \log\Bigg(\Big\|\sum_{i_1i_2i_3/i_\ell} [\mathcal{X}^{(l)}*\mathcal{P}]_{i_1i_2i_3} - \textbf{p}^{(\ell)}\Big\|\Bigg)
\end{equation*}
for $\ell=1,2,3$. Both the convergence and its linear convergence rate in terms of marginal violations are depicted in Figure~\ref{fig:convergence1}.
\begin{figure}[t]
\centering
\includegraphics[width=9cm]{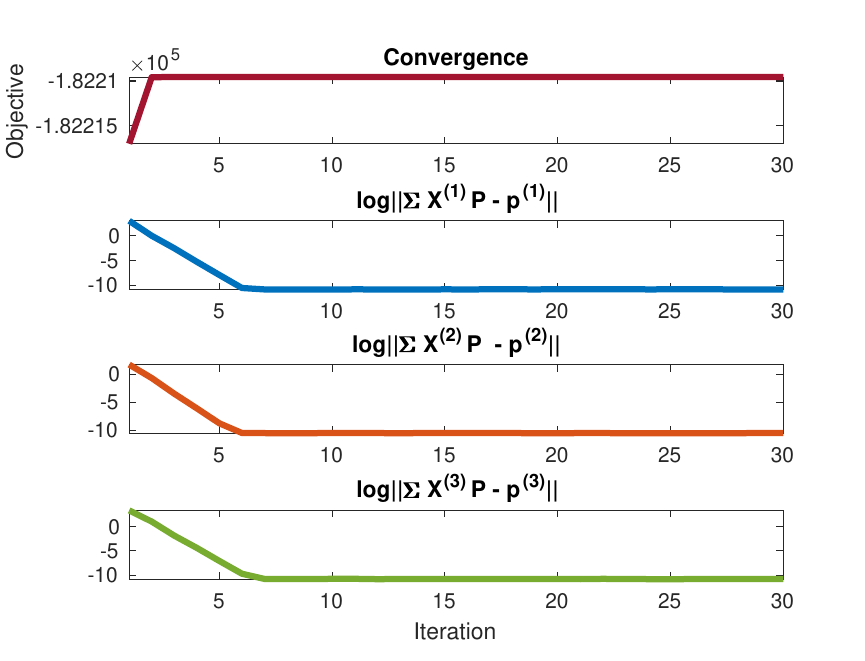}
\caption{The convergence of the proposed algorithm and the linear convergence rate in terms of marginal violations.}
\label{fig:convergence1}
\end{figure}

\section{Conclusion}\label{sec:conclude}
In this article, we introduced a generalized Sinkhorn algorithm with directional sign indicators, where each sign pattern serves as an initial guess along with the corresponding marginal, to quantitatively learn edge weights in a (higher-order) network. We also investigated the convergence and convergence rate of the proposed method. The current framework has demonstrated efficiency in the maximum likelihood estimation of (high-order) networks. Furthermore, given that a prior and a posterior are high-dimensional arrays, the problem may encounter the ``curse of dimensionality,'' as both memory and computation complexity increase exponentially with the size of the problem. To mitigate this challenge, it is worthwhile to apply tensor decomposition techniques.
\black


\bibliographystyle{plain}
\bibliography{references}

\end{document}